\documentclass{amsart}
\usepackage{amsfonts,amsthm,amsmath}
\usepackage[usenames,dvipsnames]{color}
\usepackage{xypic}
\usepackage[a4paper]{geometry}

\DeclareMathOperator*{\Res}{Res}

\theoremstyle{plain}
	\newtheorem{thm}{Theorem}
	
	\newtheorem{lemma}[thm]{Lemma}
	\newtheorem{prop}[thm]{Proposition}
	
	\newtheorem*{corollary*}{Corollary}
	\newtheorem*{mainlemma}{Main Lemma}
\theoremstyle{definition}
	\newtheorem{definition}[thm]{Definition}
	\newtheorem*{definition*}{Definition}

\theoremstyle{remark}
	\newtheorem*{rk*}{Remark}
	\newtheorem{rk}{Remark}
	\newtheorem*{example}{Example}

\title{Computation of Jeffrey-Kirwan residues using Gr\"obner bases}
\date{}
\author{Zsolt Szil\'agyi}
\address{Section de Math\'ematiques, Universit\'e de Gen\`eve, Gen\`eve, Suisse.}
\email{zsolt.szilagyi@unige.ch}
\begin{document}
\maketitle

%:Abstract
\begin{abstract}
The Jeffrey-Kirwan residue is a powerful tool for computation of intersection numbers or volume of symplectic quotients. In this article, we give an algorithm to compute it using Gr\"obner bases. Our result is parallel to that of \cite{CD} for Grothendieck residues.
\end{abstract}

\section{Introduction}

The Jeffrey-Kirwan residue was introduced in \cite{JKi1} and it is a powerful tool to compute intersection numbers or symplectic volume of symplectic quotients. There are several ways to compute it such as iterated residues, inverse Laplace transforms or nested sets  (\cite{JKi2}, \cite{JKo}, \cite{BV}, \cite{SzV}, \cite{dCP}).

The aim of this article is to give an algorithm for computation Jeffrey-Kirwan residue using Gr\"obner bases which can be implemented as a computer program.

The contents of the article are as follows. In Section \ref{Sec2} we quickly review basic notions related to Gr\"obner bases. In Section \ref{Sec3} we recall an algorithm for computation of Grothendieck residue from \cite{CD} which is similar to ours.
 In Section \ref{Sec4} we generalize ideas from previous section which will be the core of our algorithm. In Section \ref{Sec5} we recall the definition and properties of Jeffrey-Kirwan residue. In Section \ref{Sec6} we apply the results of Section \ref{Sec2} in the case of Jeffrey-Kirwan residue and at the end of the section we give an algorithm to compute it using Gr\"obner bases.
\vskip1ex
%:Acknowledgement
\noindent {\bf Acknowledgement.}
The author is grateful to Andr\'as Szenes and Mich\`ele Vergne for discussions and useful comments. The support of FNS grant 132873 is gratefully acknowledged.

\section{Gr\"obner bases and the division algorithm}\label{Sec2}

In this section we recall some notions and results from the book \cite{CLO} about Gr\"obner bases and division of multivariate polynomials. Let $\mathbb{K}$ be a field ($\mathbb{R}$ or $\mathbb{C}$) and consider the polynomial ring $\mathbb{K}[x_{1},\ldots,x_{r}]$. We use the notation $x^{a}$, $a=(a_{1}, \ldots, a_{r})\in\mathbb{N}^{r}$ for the monomial $x_{1}^{a_{1}}\cdots x_{r}^{a_{r}}$, as special case $x^{0}=1$. 
\begin{definition*}
%\note{Def.1 pp.53 and Cor. 6 pp.70 [CLO].}%
A {\it monomial order} $``>"$ on $\mathbb{K}[x_{1},\ldots, x_{r}]$ is an order on monomials with properties:
	{
	\renewcommand{\labelenumi}{(\roman{enumi})}
	\begin{enumerate}
\item it is a total order,
\item $x^{a} > x^{b}$ implies that $x^{a+c} > x^{b+c}$ for all $c\in \mathbb{N}^{r}$,
\item it is a well-ordering on $\mathbb{N}^{r}$, that is every non-empty subset of $\mathbb{N}^{r}$ has a smallest element.
	\end{enumerate}
	}
\end{definition*}
\begin{example}
The lexicographic order is defined as follows. Let $a=(a_{1},\ldots, a_{r})$ and $b=(b_{1},\ldots,b_{r})$ in $\mathbb{N}^{r}$. Then $x^{a}>_{\textnormal{lex}}x^{b}$ if in $a-b=(a_{1}-b_{1},\ldots, a_{r}-b_{r})$ the leftmost non-zero entry is positive. 
\end{example}
	\begin{definition*}
Fix a monomial order and let $\displaystyle f=\sum_{a\in A}c_{a}x^{a}$ be a non-zero polynomial in $\mathbb{K}[x_{1},\ldots,x_{r}]$. Let $x^{m}$ be the greatest monomial of the set $\{x^{a}\ |\ a\in A\}$ with respect to the fixed monomial order. Then $\textsc{Lm}(f)=x^{m}$ is called the {\it leading monomial} of $f$. Moreover, we call $\textsc{Lt}(f)=c_{m}x^{m}$ the {\it leading term} of $f$ and $\textsc{Lc}(f)=c_{m}$ is called the {\it leading coefficient} of $f$. Obviously $\textsc{Lt}(f)=\textsc{Lc}(f)\cdot \textsc{Lm}(f)$.
If $I$ is an ideal in $\mathbb{K}[x_{1},\ldots, x_{r}]$ then the {\it ideal of leading terms} $\textsc{Lt}(I)=\langle \textsc{Lt}(f)  \,|\,  f\in I  \rangle$ is the ideal generated by the leading terms of all polynomials in $I$.  
	\end{definition*}
	\begin{definition*}
A finite subset $G$ of $I$ is called a {\it Gr\"obner basis} of $I$ with respect to the chosen monomial order if $\textsc{Lt}(I)=\langle \textsc{Lt}(g)  \,|\,  g\in G \rangle$.
	\end{definition*}
\noindent
In particular, $G$ generates the ideal $I$ (\cite{CLO}, I.\S5 Corollary 6). 
It is not unique since any finite subset of $I$ containing a Gr\"obner basis is also a Gr\"obner basis. However, for each polynomial ideal with a monomial order there is a unique {\it reduced Gr\"obner basis} $G$ of $I$ with properties: 
{
\renewcommand{\labelenumi}{(\roman{enumi})}
\begin{enumerate}
\item $\textsc{Lc}(g)=1$ for all $g\in G$,
\item for all $g\in G$ no monomial of $g$ belongs to $\big\langle \textsc{Lt}(G\setminus\{g\}) \big\rangle$.
\end{enumerate}
}%
\noindent
Gr\"obner bases of a polynomial ideal can be computed by Buchberger's algorithm (\cite{CLO}, $\textnormal{ II.\S 7}$).

Next, we recall the division of multivariate polynomials  (\cite{CLO},  $\textnormal{II}.\S 3).$ We fix a monomial order on $\mathbb{K}[x_{1},\ldots,x_{r}]$ and let $F=\{f_{1},\ldots,f_{s}\}$ be an ordered $s$-tuple of polynomials. Using the following division algorithm any polynomial $f$ can be written as $f=a_{1}f_{1} + \ldots + a_{s}f_{s}+R$ with $a_{1},\ldots,a_{s},R\in \mathbb{K}[x_{1},\ldots, x_{r}]$ such that no monomials of $R$ is divisible by any of $\textsc{Lt}(f_{1}),\ldots,\textsc{Lt}(f_{s})$. The polynomial $R$ is called the {\it remainder} of the division of $f$ by $F$. 
The division algorithm goes as follows. In the beginning we set $R$ to $f$ and $a_{1},\ldots, a_{s}$ to $0$. Suppose we have written 
$$
f=a_{1}f_{1}+\ldots + a_{s}f_{s} + R
$$ 
for some $a_{1},\ldots,a_{s},R$. Choose the greatest monomial $M$ in $R$ (with respect to the fixed monomial order) which is divisible by $\textsc{Lt}(f_{i})$ with smallest $i$. Then set the new $a_{i}$ to $\displaystyle a_{i} + \frac{\textsc{Lt}(M)}{\textsc{Lt}(f_{i})}$ and the new $R$ to $\displaystyle R-\frac{\textsc{Lt}(M)}{\textsc{Lt}(f_{i})}f_{i}$.
The algorithm ends in finite steps since the order of monomials in $R$ which are divisible by any of $\textsc{Lt}(f_{1}),\ldots,\textsc{Lt}(f_{s})$ strictly decrease and as a consequence of well-ordering every strictly decreasing sequence eventually terminates.

The remainder of the division in general depends on the way the elements of $F$ are listed. Moreover, it is possible that the remainder is non-zero even if $f\in \langle f_{1}, \ldots, f_{s} \rangle$. Such an example is given in Example 5  of \cite{CLO}, $\textnormal{II}.\S 3$. Nevertheless, if $F=\{f_{1},\ldots, f_{s}\}$ is a Gr\"obner basis of the ideal $\langle f_{1},\ldots, f_{s} \rangle$ then the remainder $R$ does not depend on the way the elements of $F$ are listed. 
	\begin{definition*}
Fix a monomial order. The {\it normal form} $N_{I}(f)$ of $f$ with respect to the ideal $I$ is a polynomial such that $f-N_{I}(f) \in I$ and no monomial of $N_{I}(f)$ is contained in $\textsc{Lt}(I)$. 
	\end{definition*}
	\begin{prop}\label{A}
If $G=\{g_{1},\ldots,g_{s}\}$ is a Gr\"obner basis of the ideal $I$ then the remainder of the division of $f$ by $G$ is equal to the normal form of $f$, that is $R=N_{I}(f)$. Moreover, $f\in I$ if and only if $R=0$ (compare \cite{CLO}, $ \textnormal{II}.\S 6$, Proposition 1 and Corollary 2).
	\end{prop}
	\begin{rk}\label{Rk0}
If $I$ is a homogeneous polynomial ideal then for any monomial order we can choose a Gr\"obner basis $G=\{g_{1},\ldots,g_{s}\}$ with $g_{i}$ homogeneous polynomials. Moreover, by the division algorithm 
the remainder of the division of a homogeneous polynomial $f$ by $G$ will be homogeneous of the same degree as $f$. Hence, $N_{I}(f)$ is homogeneous of the same degree as $f$.
	\end{rk}
	\begin{lemma}\label{Lem0}
Let $I$ be a homogeneous polynomial ideal with a monomial order. Then $I$ and $\textsc{Lt}(I)$ are isomorphic as graded vectors spaces, i.e. $\dim I_{d} = \dim \textsc{Lt}(I)_{d}$ for all $d$ (\cite{CLO}, $ \textnormal{IX}.\S 3$, Proposition 4).
	\end{lemma}

\section{Review of Grothendieck residue case}\label{Sec3}
In this section we review an algorithm from \cite{CD} for computation of Grothendieck residue with Gr\"obner bases. Let $H, P_{1}, \ldots, P_{r} \in\mathbb{C}[x_{1},\ldots,x_{r}]$ be polynomials of degree $d, d_{1},\ldots,d_{r}$, respectively. 
	\begin{definition*}
Suppose that $0 \in \mathbb{C}^{r}$ is the only common zero of $P_{1},\ldots, P_{r}$ in an open neighborhood $U$ of it. The {\it Grothendieck residue} is defined as 
$$
\textnormal{Res}_{0}\left( \frac{H}{ P_{1} | \ldots | P_{r}} \right) = 
\frac{1}{(2\pi\sqrt{-1})^{r}}\int_{\Gamma(\varepsilon)}\frac{H(x)}{P_{1}(x)\cdots P_{r}(x)} \, \textnormal{d} x_{1}\ldots \textnormal{d} x_{r},
$$
where $\Gamma(\varepsilon)=\{ x\in U \,\big|\, |P_{i}(x)|=\varepsilon_{i}>0,\ i=1,\ldots,r \}$ and it is oriented by $\textnormal{d}(\arg P_{1}(x)) \wedge \ldots \wedge \textnormal{d}(\arg P_{r}(x))$.
	\end{definition*}
\noindent The Grothendieck residue has the following properties.
{%
\renewcommand{\labelenumi}{(\roman{enumi})}%
	\begin{enumerate}%
	\item It is independent on small $\varepsilon=(\varepsilon_{1},\ldots,\varepsilon_{r})$.
	
	\item If $P_{1},\ldots, P_{r}$ and $H$ are homogeneous polynomials then  using rescaling $x\mapsto \lambda x$ we can see that the residue $\displaystyle\textnormal{Res}_{0}\left( \frac{H}{P_{1}|\ldots |P_{r}} \right)=0$ if $d \neq \sum_{i=1}^{r}(d_{i}-1)$. 

 \item (Local duality) $\displaystyle\textnormal{Res}_{0} \left( \frac{H K}{P_{1}|\ldots|P_{r}} \right) = 0$ for all $K\in \mathbb{C}[x_{1},\ldots,x_{r}]$ if and only if $H$ lies in the ideal $\langle P_{1},\ldots, P_{r} \rangle$. 
 
\item (Transformation law) Let $Q_{1},\ldots,Q_{r}\in \mathbb{C}[x_{1},\ldots,x_{r}]$ such that $0$ is their only common zero locally. 
If we can write $Q_{i}=\sum_{j=1}^{r}a_{ij} P_{j}$ then
%and 
$$
\textnormal{Res}_{0} \left( \frac{H}{P_{1}|\ldots | P_{r}} \right) = \textnormal{Res}_{0} \left( \frac{H\det \big( [a_{ij}]_{i,j=1}^{r} \big)}{Q_{1}|\ldots |Q_{r}} \right).
$$
\item (Euler-Jacobi vanishing theorem) Let $ P_{1},\ldots, P_{r}$ such that	 $0\in\mathbb{C}^{r}$ is the only common zero. If $d < \sum_{i=1}^{r} (d_{i}-1)$ then
$\displaystyle\textnormal{Res}_{0}\left( \frac{H}{P_{1}|\ldots|P_{r}} \right) = 0.$
\end{enumerate}
}
As noted in \cite{CD} $\S 1.5.6$ the Grothendieck residue can be computed using normal forms as follows. Suppose that $P_{1},\ldots, P_{r}$ are homogeneous with only common zero $0\in\mathbb{C}^{r}$.
{
\renewcommand{\labelenumi}{(\arabic{enumi})}
	\begin{enumerate}
	\item Compute a Gr\"obner basis of the ideal $I=\langle P_{1},\ldots,P_{r} \rangle$.
	\item Compute normal forms $N_{I}(H)$ of $H$ and $N_{I}(\Delta)$ of $\Delta = \det \big( [a_{ij}]_{i,j=1}^{r} \big)$, where $P_{i}=	\sum_{j=1}^{r}a_{ij}x_{j}$ for all $i=1,\ldots, r$. 
	\item $\displaystyle\textnormal{Res}_{0} \left( \frac{H}{P_{1}|\ldots | P_{r}} \right) = \frac{N_{I}(H)}{N_{I}(\Delta)}$, since $\displaystyle\textnormal{Res}_{0} \left( \frac{\Delta}{P_{1}|\ldots | P_{r}} \right) = \textnormal{Res}_{0} \left( \frac{1}{x_{1}|\ldots | x_{r}} \right) $ $ = 1$ by the transformation law.
	\end{enumerate}
	}
\noindent
Let $p_{1},\ldots,p_{r}\in \mathbb{C}[x_{1},\ldots,x_{r}]$ be (non-homogeneous) polynomials of degree $d_{1},\ldots,d_{r}$. Denote $q_{i}$ the degree $d_{i}$ part of $p_{i}$ and assume that $0\in \mathbb{C}^{r}$ is the only common zero of $q_{1},\ldots,q_{r}$. Let 
$$
P_{i}(x_{0},x_{1},\ldots,x_{r})=x_{0}^{d_{i}} \cdot p_{i}\left(\frac{x_{1}}{x_{0}},\ldots, \frac{x_{r}}{x_{0}}\right)\in \mathbb{C}[x_{0},x_{ 1},\ldots,x_{r}]
$$ be the homogenization of $p_{i}$. Let $h\in\mathbb{C}[x_{1},\ldots,x_{r}]$ be a degree $d$ polynomial. If $d< \sum_{i=1}^{r}(d_{i}-1)$ then $\displaystyle\textnormal{Res}_{0} \left( \frac{h}{p_{1}|\ldots|p_{r}} \right)=0$ by Euler-Jacobi theorem. Therefore, suppose that $d\geq \sum_{i=1}^{r}(d_{i}-1)$ and let 
$$
H(x_{0},x_{1},\ldots,x_{r})=x_{0}^{d} \cdot h\left( \frac{x_{1}}{x_{0}},\ldots, \frac{x_{r}}{x_{0}} \right)
$$
be the homogenization of $h$. Moreover, let $P_{0}(x_{0},\ldots,x_{r})=x_{0}^{d_{0}}$ with $d_{0}=d+1 - \sum_{i=1}^{r} (d_{i}-1)$. Then
$$
\textnormal{Res}_{0} \left( \frac{h}{p_{1} | \ldots | p_{r}} \right)
=
\textnormal{Res}_{0}\left( \frac{H}{P_{0}|P_{1}|\ldots|P_{r}} \right)
$$
and we can apply the above algorithm to compute it.
\section{The main lemma}\label{mainidea}\label{Sec4}
Let $\mathbb{K}$ be a field and denote $\mathbb{K}[x]$ the polynomial ring $\mathbb{K}[x_{1},\ldots,x_{r}]$. Denote $\mathbb{K}[x]_{d}$ the vector space of homogeneous polynomials of degree $d$. Consider a $\lambda: \mathbb{K}[x]_{d}\to \mathbb{K}$ non-zero linear functional and let $I\subset \mathbb{K}[x]$ be a homogeneous ideal such that 
$$
\ker \lambda = I\cap \mathbb{K}[x]_{d}.
$$ 
%For an ideal $I$ we introduce notation $I_{d}=I\cap\mathbb{K}[x]_{d}$.

%
	\begin{mainlemma}\label{B}
Let $\Delta\in \mathbb{K}[x]_{d}$ such that $\lambda(\Delta)\neq 0$. Then $N_{I}(\Delta)\neq 0$ and for all $P\in \mathbb{K}[x]_{d}$ we have
$$\frac{\lambda(P)}{\lambda(\Delta)}=\frac{N_{I}(P)}{N_{I}(\Delta)}.$$ 
	\end{mainlemma}
	\begin{proof}
Since $\lambda(\Delta)\neq 0$, $\Delta$ is not contained in $I$, thus  $N_{I}(\Delta)$ is a non-zero homogeneous polynomial of degree $d$ by Proposition \ref{A} and Remark \ref{Rk0}.  
Since  $ I_{d}=I\cap \mathbb{K}[x]_{d}$ is $1$-codimensional in $\mathbb{K}[x]_{d}$, the vector space $\textsc{Lt}(I)_{d}$ is also one codimensional by Lemma \ref{Lem0}. 
%Indeed, if $M_{1}, M_{2}\in \mathbb{K}[x]_{d}\setminus \textsc{Lt}(I)_{d}$ monomials then $\lambda(M_{2})\neq 0$ and there is $c\in \mathbb{K}$ such that $M_{1}-cM_{2}\in \ker \lambda = I_{d}$, hence $M_{1}$ or $M_{2}$ belongs to $\textsc{Lt}(I)_{d}$.
%
Therefore, for all $P\in \mathbb{K}[x]_{d}$ the normal form $N_{I}(P)$ is a multiple of the unique monomial $M$ of degree $d$ which is not contained in  $\textsc{Lt}(I)_{d}$, that is $N_{I}(P)=\textsc{Lt}(N_{I}(P))=\textsc{Lc}(N_{I}(P))\cdot M$.  
Finally, for all $P\in \mathbb{K}[x]_{d}$ we have
$$
\lambda(P) 
= \lambda\big(N_{I}(P)\big) 
=  \textsc{Lc}(N_{I}(P))\cdot \lambda\big(\textsc{Lm}(N_{I}(P))\big)
= \textsc{Lc}(N_{I}(P))\cdot \lambda\big(M\big),
$$
and
$$
\frac{\lambda(P)}{\lambda(\Delta)}
= 
\frac{\textsc{Lc}(N_{I}(P))\cdot \lambda(M)}{\textsc{Lc}(N_{I}(\Delta))\cdot \lambda(M)} 
= 
\frac{\textsc{Lc}(N_{I}(P))\cdot M}{\textsc{Lc}(N_{I}(\Delta))\cdot M} 
= 
\frac{N_{I}(P)}{N_{I}(\Delta)}.
$$
\end{proof}
\noindent
If we can find an ideal $I$ such that $I_{d}=\ker\lambda$ and if we can compute a non-zero $\lambda(\Delta)$ for some $\Delta$ then we can also compute $\lambda(P)$ for all $P\in \mathbb{K}[x]_{d}$ using the following algorithm:
{
\renewcommand{\labelenumi}{(\arabic{enumi})}
	\begin{enumerate}
	\item compute a Gr\"obner basis $G$ of $I$,
	\item compute normal forms $N_{I}(P)\>$and $N_{I}(\Delta)\>$as remainders of the division of $P$ and $\Delta$ by $G$,	
	\item finally, $\lambda(P)= \lambda(\Delta) \cdot \displaystyle\frac{N_{I}(P)}{N_{I}(\Delta)}\in \mathbb{K}$.
	\end{enumerate}
}
\noindent
%In the case of Jeffrey-Kirwan residue we have to compute the ideal $I$ and we have to find a $\Delta$.

\section{The Jeffrey-Kirwan residue}\label{Sec5}

We recall the definition and properties of Jeffrey-Kirwan residue following \cite{JKo}. Let $V$ be an $r$-dimensional real vector space and let $\mathfrak{A}=[\alpha_{1},\ldots, \alpha_{n}]$ be a collection of (not necessarily distinct) non-zero vectors in $V^{*}$. We consider $\alpha_{i}$'s as linear functions on $V$. Let $\Lambda$ be a connected component of $V\setminus \bigcup_{i=1}^{n}\alpha_{i}^{\perp}$, where $\alpha_{i}^{\perp} = \{v\in V\ |\ \alpha_{i}(v)=0 \}$. Remark that for all $i$, either $\alpha_{i}\in \Lambda^{\vee}$ or $-\alpha_{i} \in \Lambda^{\vee}$, where $\Lambda^{\vee}=\{\beta \in V^{*}\ |\ \beta(v) > 0,\ \forall v\in\Lambda\}$ is the dual cone of $\Lambda$. When $\alpha_{i}\in \Lambda^{\vee}$ we say that $\alpha_{i}$ is {\it polarized}. Let $\xi\in\Lambda$ and choose a basis $\{x_{1},\ldots, x_{r}\}$ of $V^{*}$ such that $x_{1}(\xi)=1$, and $x_{2}(\xi)=\ldots=x_{r}(\xi)=0$. Let $\varepsilon = \varepsilon_{1}x_{1}+\ldots + \varepsilon_{r}x_{r}\in V^{*}$ and let $P\in \mathbb{R}[V]$ be a polynomial. 
	\begin{definition*}
We  define
$$
\textnormal{Res}^{+}_{x_{1}} \frac{P(x) e^{\varepsilon(x)}}{\prod_{i=1}^{n}\alpha_{i}(x)} \textnormal{d} x_{1}
=
\begin{cases}
\Res\limits_{x_{1}=\infty}\dfrac{P(x)e^{\varepsilon(x)}}{\prod_{i=1}^{n} \alpha_{i}(x)}\textnormal{d} x_{1} & \textnormal{if }\varepsilon_{1}\geq 0,
\\
0 & \textnormal{if }\varepsilon_{1}<0,
\end{cases}
$$
considering $x_{2},\ldots,x_{r}$ as constants while taking the residue with respect to $x_{1}$. 
We define the Jeffrey-Kirwan residue
$$
\textnormal{JKRes}^{\Lambda} \frac{P(x) e^{\varepsilon(x)}}{\prod_{i=1}^{n}\alpha_{i}(x)} \textnormal{d} x
= 
\frac{1}{\sqrt{\det[(x_{i},x_{j})]_{i,j=1}^{r}}} 
%\cdot 
\textnormal{Res}^{+}_{x_{r}} \left( \ldots \left( \textnormal{Res}^{+}_{x_{1}}\, \frac{P(x)e^{\varepsilon(x)}}{\prod_{i=1}^{n}\alpha_{i}(x)} \textnormal{d} x_{1} \right)\ldots \right)\textnormal{d} x_{r}, 
$$
where $\det[(x_{i},x_{j})]_{i,j=1}^{r}$ is the Gram determinant with respect to a fixed scalar product on $V^{*}$.
	\end{definition*}
	\begin{definition*}
We call an element of $V^{*}$
{\it regular} (with respect to $\mathfrak{A}$) if it does not lie on any $(r-1)$- or less dimensional subspace of $V^*$ spanned by subsets of $\mathfrak{A}$.
	\end{definition*}
\noindent
The Jeffrey-Kirwan residue has the following properties (\cite{JKi2}, Proposition 3.2).
{
\renewcommand{\labelenumi}{(\arabic{enumi})}
\begin{enumerate}
\item[(P1)]\label{P1} Suppose that $\alpha_{1},\ldots, \alpha_{n}\in \Lambda^{\vee}$ and $\varepsilon$ is regular. Let $P\in \mathbb{R}[V]$ be a homogeneous polynomial of degree $d$. Then
$$
\textnormal{JKRes}^{\Lambda} \frac{P(x) e^{\varepsilon(x)}}{\prod_{i=1}^{n}\alpha_{i}(x)} \textnormal{d} x = 0
$$
unless all the following properties are satisfied
\begin{enumerate}
\item $\{\alpha_{1},\ldots, \alpha_{n}\}$ spans $V^{*}$ as vector space,
\item $d \leq n-r$,
\item $\varepsilon\in Cone(\alpha_{1},\ldots, \alpha_{n}) = \big\{ \sum_{i=1}^{r}a_{i}\alpha_{i} \,|\, a_{1},\ldots,a_{r}\geq 0 \big\} $.
\end{enumerate}
\item[(P2)]\label{P2} Suppose that $d\leq n-r$ and $P$ is homogeneous polynomial of degree $d$. Then
\begin{align*}
\textnormal{JKRes}^{\Lambda} \frac{P(x) e^{\varepsilon(x)}}{\prod_{i=1}^{n}\alpha_{i}(x)} \textnormal{d} x  
&{}= 
\sum_{k\geq0}\lim_{t\to 0^{+}} \textnormal{JKRes}^{\Lambda} \frac{\varepsilon(x)^{k} P(x) e^{t \varepsilon(x)}}{k!\,\prod_{i=1}^{n}\alpha_{i}(x)} \textnormal{d} x 
\\&
{}= 
 \textnormal{JKRes}^{\Lambda} \frac{\varepsilon(x)^{n-r-d} P(x) e^{\varepsilon(x)}}{(n-r-d)!\,\prod_{i=1}^{n}\alpha_{i}(x)} \textnormal{d} x .
\end{align*}
\item[(P3)]\label{P3} If $d=0$, $n=r$ and properties (P1)(a)-(c) are satisfied then
$$
\textnormal{JKRes}^{\Lambda}  \frac{e^{\varepsilon(x)}}{\prod_{i=1}^{r}\alpha_{i}(x)} \textnormal{d} x  = \frac{1}{|\det(\alpha_{1},\ldots,\alpha_{r})|} = \frac{1}{\sqrt{\det[(\alpha_{i},\alpha_{j})]_{i,j=1}^{r}}},
$$
where $(\alpha_{1},\ldots,\alpha_{r})$ denotes the matrix whose columns are the coordinates of $\alpha_{1}, \ldots ,\alpha_{r}$ with respect to any orthonormal basis of $V^{*}$.
\end{enumerate}
}
For $\varepsilon$ regular the Jeffrey-Kirwan residue does not depend on the choice of $\xi \in \Lambda$ and the corresponding basis $\{ x_{1}, \ldots, x_{r}\}$ of $V^{*}$.

We fix $\mathfrak{A}=[\alpha_{1},\ldots,\alpha_{n}]$ and we assume that its elements are polarized, that is $\alpha_{1},\ldots,\alpha_{n}\in \Lambda^{\vee}$ for some $\Lambda$. We also fix $\varepsilon\in V^{*}$ regular with respect to $\mathfrak{A}$. By property (P2) the residue $\displaystyle\textnormal{JKRes}^{\Lambda}\frac{P(x) e^{\varepsilon(x)}}{\prod_{i=1}^{n}\alpha_{i}(x)}\textnormal{d}x$ is determined by values for $P$ homogeneous polynomials of degree $n-r$. Thus, we are interested in computation of the linear functional $P\mapsto \displaystyle\textnormal{JKRes}^{\Lambda} \frac{P(x) e^{\varepsilon(x)}}{\prod_{i=1}^{n}\alpha_{i}(x)}\textnormal{d}x$ on $\mathbb{R}[V]_{n-r}$. Let $R_{\mathfrak{A}}=\mathbb{R}[V]_{\mathfrak{A}}$ be the vector space of rational functions of form $\displaystyle\frac{P}{\prod_{i=1}^{n}\alpha_{i}^{m_{i}}}$, where $P\in \mathbb{R}[V]$ and $m_{i}\geq0$.
A subset $J \subset \{1,\ldots,n\}$ is called {\it generating} if $\{\alpha_{j}\ |\ j\in J\}$ spans $V^{*}$ as vector space. 
We have the following partial fraction decomposition (more generally, Theorem 1 of \cite{BV}).
\begin{prop}\label{Prop1}
Let $G_{\mathfrak{A}}\subset R_{\mathfrak{A}}$ be spanned as vector space by fractions of form $\displaystyle\frac{1}{\prod_{i=1}^{n}\alpha_{i}^{m_{i}}}$ such that the set $J=\{i\ |\ m_{i}>0\}$ is generating. Moreover, let $NG_{\mathfrak{A}}\subset R_{\mathfrak{A}}$ be the vector space spanned by fractions of form $\displaystyle\frac{P}{\prod_{i=1}^{n}\alpha_{i}^{m_{i}}}$ such that the set $J=\{i \ |\ m_{i}>0 \}$ is non-generating. Then we have a direct sum decomposition
$$
R_{\mathfrak{A}} = G_{\mathfrak{A}}\oplus NG_{\mathfrak{A}}.
$$
\end{prop}

Denote $(R_{\mathfrak{A}})_{d}$ the degree $d$ part of $R_{\mathfrak{A}}$, $d\in \mathbb{Z}$. In particular, $(R_{\mathfrak{A}})_{d} = (G_{\mathfrak{A}})_{d} \oplus (NG_{\mathfrak{A}})_{d}$ and $(G_{\mathfrak{A}})_{-r}$ is spanned by fractions $\displaystyle\frac{1}{\alpha_{i_{1}}\cdots \alpha_{i_{r}}}$ such that $\{\alpha_{i_{1}},\ldots,\alpha_{i_{r}}\}$ is a basis of $V^{*}$. In general, these fractions are not linearly independent, for example $\displaystyle\frac{1}{x(x+y)}=\frac{1}{xy}-\frac{1}{(x+y)y}$.
Nevertheless, following \cite{SzV} we can define
	\begin{definition}\label{Def1}
Fix a scalar product on $V^{*}$ and consider the linear functional $\textnormal{JK}_{\varepsilon}$ on $R_{\mathfrak{A}}$ which vanishes on $R_{\mathfrak{A}}\setminus (G_{\mathfrak{A}})_{-r}$ and on $(G_{\mathfrak{A}})_{-r}$ is defined by
$$
\textnormal{JK}_{\varepsilon}\left( \frac{1}{\alpha_{i_{1}}\cdots \alpha_{i_{r}}} \right) 
= \begin{cases}
\dfrac{1}{\sqrt{\det[(\alpha_{i_{k}},\alpha_{i_{l}})]_{k,l=1}^{r}}}
 & 
 \textnormal{if }\varepsilon\in Cone(\alpha_{i_{1}},\ldots,\alpha_{i_{r}}),
\\
0 & \textnormal{otherwise.}
\end{cases} 
$$
	\end{definition}
By Proposition \ref{Prop1} and properties (P1), (P3) if $P\in \mathbb{R}[V]$ homogeneous polynomial of degree $n-r$ then
\begin{equation}\label{Eq10}
\textnormal{JK}_{\varepsilon} \left( \frac{P}{\prod_{i=1}^{n}\alpha_{i}} \right) 
= 
\textnormal{JKRes}^{\Lambda}\frac{P(x) e^{\varepsilon(x)}}{\prod_{i=1}^{n}\alpha_{i}(x)}\textnormal{d}x,
\end{equation}
consequently $\textnormal{JK}_{\varepsilon}$ is well-defined.

	\begin{rk*}
 In the case of $V^*=\mathbb{R}^r$ with the standard scalar product, $\mathfrak{A}=[\alpha_1,\ldots,\alpha_n]$ a generating subset of $\mathbb{Z}^r$ and 
$\varepsilon$ regular the Jeffrey-Kirwan residues can be interpreted as intersection numbers on the toric variety $X_{\mathfrak{c}}(\mathfrak{A})$: 
$
\displaystyle\textnormal{JK}_\varepsilon \left( \frac{P}{\prod_{i=1}^n \alpha_i} \right)
=
\int_{X_{\mathfrak{c}}(\mathfrak{A})}\chi(P),
$ 
where $\mathfrak{c}$ is the connected component of regular elements containing $\varepsilon$ and 
$
\chi: \mathbb{R}[x_1,\ldots,x_r]\to H^{\bullet}(X_{\mathfrak{c}}(\mathfrak{A}))
$ 
is a degree preserving ring homomorphism ($\deg x_i =2$) as explained in \cite{SzV}, $\S\textnormal{1}$ and $\S\textnormal{2}$.  
\end{rk*}

%\newpage

%
\section{Computation of Jeffrey-Kirwan residue using Gr\"obner bases}\label{Sec6}
To apply the result of Section \ref{Sec4} for the Jeffrey-Kirwan residue we have to find a homogeneous ideal $\mathcal{I}_{\mathfrak{A},\varepsilon}$ such that for all $P\in \mathbb{R}[V]_{n-r}$ we have $\displaystyle\textnormal{JK}_{\varepsilon}\left(\frac{P}{\prod_{i=1}^{n}\alpha_{i}}\right)=0$ if and only if $P\in \mathcal{I}_{\mathfrak{A},\varepsilon}$, and a $\Delta \in \mathbb{R}[V]_{n-r}$ for which we can compute $\displaystyle\textnormal{JK}_{\varepsilon}\left(\frac{\Delta}{\prod_{i=1}^{n}\alpha_{i}}\right) \neq 0$.
Recall that if $\varepsilon \notin Cone(\alpha_{1},\ldots,\alpha_{n})$ then $\displaystyle\textnormal{JK}_{\varepsilon}\left(\frac{P}{\prod_{i=1}^{n}\alpha_{i}}\right)=0$ for all $P$ by (\ref{Eq10}) and property (P1). 
We will assume that $\varepsilon \in Cone(\alpha_{1},\ldots,\alpha_{n})$, hence there is a basis $\{ \alpha_{i_{1}},\ldots,\alpha_{i_{r}}\}$ such that $\varepsilon\in Cone(\alpha_{i_{1}},\ldots,\alpha_{i_{r}})$, thus $\displaystyle\textnormal{JK}_{\varepsilon}\left(\frac{1}{\alpha_{i_{1}}\cdots \alpha_{i_{r}}}\right) = \frac{1}{\sqrt{\det[(\alpha_{i_{k}},\alpha_{i_{l}})]_{k,l=1}^{r}}} \neq 0$. 
Therefore, let $\displaystyle\Delta = \frac{\prod_{i=1}^{n}\alpha_{i}}{\alpha_{i_{1}}\cdots \alpha_{i_{r}}}$. 
%The following theorem describes an ideal $\mathcal{I}_{\mathfrak{A},\varepsilon}$.
Denote $N=\{1,\ldots,n\}$.

\begin{thm}\label{C}
Let $P\in \mathbb{R}[V]_{n-r}$. Then $\displaystyle\textnormal{JK}_{\varepsilon}\left( \frac{P}{\prod_{i=1}^{n}\alpha_{i}} \right) = 0$ if and only if $P$ belongs to the homogeneous ideal $\mathcal{I}_{\mathfrak{A},\varepsilon} = \big\langle \prod_{j\in J}\alpha_{j}  \,\big|\,  \varepsilon\notin Cone(\alpha_{i} \,|\, i\in N \setminus J)\big\rangle$. More geometrically, $\mathcal{I}_{\mathfrak{A},\varepsilon}$ is generated by all $\prod_{\alpha_{i}\in H^{+}}\alpha_{i}$ where $H^{+}$ is the open half-space containing $\varepsilon$ of a hyperplane $H$ spanned by a subset of $\mathfrak{A}$.
\end{thm}
Before the proof of the theorem we need some preparation. Denote $\mathcal{H}(\mathfrak{A})$ the set of hyperplanes in $V^{*}$ which are spanned by subsets of $\mathfrak{A}$. Let $\mathcal{C}(\mathfrak{A})$ be the set of connected components of $Cone(\mathfrak{A})\setminus \bigcup_{H\in \mathcal{H}(\mathfrak{A})}H$ and its elements will be called {\it chambers}. For $\varepsilon$ regular with respect to $\mathfrak{A}$ denote $\mathfrak{c}_{\varepsilon}\in \mathcal{C}(\mathfrak{A})$ the chamber containing $\varepsilon$. Denote $\mathcal{B}(\mathfrak{A})$ the set of all bases $\sigma$ of $V^{*}$ such that $\sigma\subset \mathfrak{A}$ and let $\mathcal{B}(\mathfrak{A},\mathfrak{c}) = \{ \sigma \in \mathcal{B}(\mathfrak{A})  \,|\,  \mathfrak{c} \not\subset Cone(\sigma) \}$. To any basis $\sigma \in \mathcal{B}(\mathfrak{A})$ we associate a fraction $\displaystyle\Phi_{\sigma} = \frac{1}{\prod_{\alpha \in \sigma}\alpha} \in (G_{\mathfrak{A}})_{-r}$.

	\begin{rk*} 
\noindent
\begin{enumerate}
\item For any $\sigma \in \mathcal{B}(\mathfrak{A})$ we have $\mathfrak{c} \cap Cone(\sigma)\neq \emptyset$ if and only if $\mathfrak{c} \subset Cone(\sigma)$. 

\item  Since we have $\varepsilon\in Cone(\alpha_{i_{1}},\ldots,\alpha_{i_{r}})$ exactly when $\mathfrak{c}_{\varepsilon} \subset Cone(\alpha_{i_{1}},\ldots,\alpha_{i_{r}})$ for any basis $\{\alpha_{i_{1}},\ldots,\alpha_{i_{r}}\}$, therefore the $\displaystyle\textnormal{JK}_{\varepsilon}\left( \frac{P}{\prod_{i=1}^{n} \alpha_{i}}\right)$ depends only on the chamber $\mathfrak{c}_{\varepsilon}$, not on the particular vector $\varepsilon$ by Propostion \ref{Prop1} and Definition \ref{Def1}. 

\end{enumerate}
	\end{rk*}

The main tool in the proof of Theorem \ref{C} will be the following proposition.

\begin{prop}\label{Prop2}
Let $\mathfrak{c} \in \mathcal{C}(\mathfrak{A})$ and fix $\tau \in \mathcal{B}(\mathfrak{A})$ such that $\mathfrak{c} \subset  Cone(\tau)$. Then for any $\rho \in \mathcal{B}(\mathfrak{A})$ we have a decompostion of fractions 
$$
\Phi_{\rho} = a_{\tau} \Phi_{\tau} + \sum_{\sigma \in \mathcal{B}(\mathfrak{A},\mathfrak{c})}a_{\sigma} \Phi_{\sigma}.
$$
\end{prop}
We will prove it by induction on $\dim V^{*}$ and wall-crossing, but first we prove it in two particular cases.

	\begin{lemma}\label{Lem1} 
Let $\beta_{0},\ldots,\beta_{r}\in \mathfrak{A}$. 
If $\mathfrak{c}\subset  Cone(\beta_{1},\ldots,\beta_{r})$ then there is a unique $l>0$ such that $\mathfrak{c} \subset Cone(\beta_{0},\ldots,\widehat{\beta_{l}},\ldots,\beta_{r})$. In particular, if $\beta_{0} = \sum_{i=1}^{r}b_{i}\beta_{r}$ then 
	\begin{equation}\label{Eq0}
\frac{1}{\beta_{0} \cdots \widehat{\beta_{l}}\cdots\beta_{r}}
 =  
 \frac{b_{l}^{-1}}{\beta_{1} \cdots \beta_{r}} 
 - 
 \sum_{1\leq i \neq l \leq r}  \frac{b_{i}b_{l}^{-1}}{\beta_{0}\cdots \widehat{\beta_{i}}\cdots\beta_{r}},
	\end{equation}
and $\{\beta_{0},\ldots, \widehat{\beta_{i}},\ldots,\beta_{r}\} \in \mathcal{B}(\mathfrak{A},\mathfrak{c})$ for all $1\leq i\neq l \leq r$.
	\end{lemma}

	\begin{proof}
Let $\varepsilon \in \mathfrak{c}$. Then $\varepsilon\in Cone(\beta_{1},\ldots, \beta_{r})$ and $\varepsilon$ is regular with respect to $\mathfrak{A}$. It follows that $\{\beta_{1},\ldots,\beta_{r}\}$ is a basis of $V^{*}$, hence we can write $\varepsilon = \sum_{i=1}^{r} e_{i}\beta_{i}$ with $e_{1},\ldots,e_{r}>0$ and $\beta_{0}= \sum_{i=1}^{r} b_{i}\beta_{i}$. For any $l\in \{1,\ldots, r\}$ with $b_{l}\neq 0$ we have
$$
\varepsilon = \frac{e_{l}}{b_{l}}\beta_{0} + \sum_{1\leq i\neq l \leq r} \frac{e_{i}b_{l}-e_{l}b_{i}}{b_{l}} \beta_{i}.
$$
Thus, $\varepsilon \in Cone(\beta_{0},\ldots,\widehat{\beta_{l}},\ldots,\beta_{r})$ exactly when
	\begin{equation}\label{Eq1}
\frac{e_{l}}{b_{l}} \geq 0 \qquad \textnormal{and} \qquad \frac{e_{i}b_{l}-e_{l}b_{i}}{b_{l}} \geq 0
	\end{equation}
for all $i=1,\ldots,\widehat{l},\ldots,r$.
Since $\varepsilon$ is regular with respect to $\{\beta_{0},\ldots,\beta_{r}\}$ we have that $\dfrac{e_{l}}{b_{l}}\neq 0$ and  $\dfrac{e_{i}b_{l}-e_{l}b_{i}}{b_{l}} \neq 0$ for all $i=1,\ldots,\widehat{l},\ldots, r$, and consequently,
	\begin{equation}\label{Eq2}
\frac{b_{l}}{e_{l}}\neq \frac{b_{i}}{e_{i}}
	\end{equation}
for all $i=1,\ldots,\widehat{l},\ldots,r$. Hence inequalities (\ref{Eq1}) are equivalent to
	\begin{equation}\label{Eq3}
b_{l}>0 \qquad \textnormal{and} \qquad \frac{b_{l}}{e_{l}} > \frac{b_{i}}{e_{i}}
	\end{equation}
for all $i=1,\ldots,\widehat{l},\ldots,r$, because $e_{1},\ldots,e_{r}>0$. 

We may suppose that $\displaystyle\frac{b_{1}}{e_{1}} \leq \ldots \leq \frac{b_{r-1}}{e_{r-1}} \leq \frac{b_{r}}{e_{r}}$. Since $\beta_{0},\ldots,\beta_{r}\in \mathfrak{A}$ are polarized there exists $\xi\in V$ such that $\beta_{0}(\xi), \ldots, \beta_{r}(\xi)>0$ and from $0 < \beta_{0}(\xi) = \sum_{i=1}^{r} b_{i}\beta_{i}(\xi)$ follows that there are $i$ with $b_{i}>0$. In particular, we must have $b_{r}>0$, because $e_{1},\ldots,e_{r}>0$. Moreover, by (\ref{Eq2}) we have $\displaystyle\frac{b_{i}}{e_{i}}< \frac{b_{r}}{e_{r}}$ for all $i=1,\ldots,r-1$, therefore $\varepsilon\in Cone(\beta_{0},\ldots,\widehat{\beta_{l}},\ldots,\beta_{r})$ if and only if $l=r$.

Finally, dividing the relation $\displaystyle\beta_{l} = \frac{\beta_{0}}{b_{l}} - \sum_{1\leq i \neq l \leq r} \frac{b_{i}}{b_{l}}\beta_{i}$ by $\beta_{0}\cdots \beta_{r}$ we get (\ref{Eq0}) and from the first part follows that $\{\beta_{0}, \ldots, \widehat{\beta_{i}},\ldots, \beta_{r}\}\in \mathcal{B}(\mathfrak{A},\mathfrak{c})$ for all $i=1,\ldots,\widehat{l},\ldots,r$.
 \end{proof}

\begin{lemma}\label{Lem2}
If Proposition \ref{Prop2} holds when $\dim V^{*}<r$ then it also holds if $\dim V^{*}=r$ and $\mathfrak{c}\in \mathcal{C}(\mathfrak{A})$ 
is a chamber such that its closure $\overline{\mathfrak{c}}$ intersects a face $F$ of $Cone(\mathfrak{A})$ in an $(r-1)$-dimensional polyhedral cone.
	\end{lemma}

	\begin{rk}\label{Rk1}
Let $w$ be an $(r-1)$-dimensional face of $\overline{\mathfrak{c}}$ and let $W$ be its supporting hyperplane. If $\sigma \in \mathcal{B}(\mathfrak{A})$ is in the same closed half-space of $W$ as $\mathfrak{c}$ then $w\subset Cone(\sigma)$ if and only if $\mathfrak{c}\subset Cone(\sigma)$.
	\end{rk}

	\begin{proof}[Proof of Lemma \ref{Lem2}]
Let $W$ be the supporting hyperplane of $F$. Let $w = \overline{\mathfrak{c}}\cap W$, which is an $(r-1)$-dimensional polyhedral cone and its relative interior is contained in a chamber $\mathfrak{c}_{W} \in \mathcal{C}(\mathfrak{A}\cap W)$. 
Let $\rho \in \mathcal{B}(\mathfrak{A})$ such that $\mathfrak{c} \subset Cone(\rho)$, thus $w = \overline{\mathfrak{c}}\cap W \subset Cone(\rho \cap W)$. Moreover, $Cone(\rho \cap W)$ is an $(r-1)$-dimensional face of the simplicial cone $Cone(\rho)$, hence $\rho\cap W$ is a basis of $W$.
Similarly, we have $w \subset Cone(\tau \cap W)$ and $\tau \cap W$ is a basis of $W$.
In particular, $\mathfrak{c}_{W}$ is contained in both $Cone(\tau \cap W)$ and $Cone(\rho\cap W)$. By assumption that Proposition \ref{Prop2}  holds on $W$ we get
$$
\Phi_{\rho \cap W} = a_{\tau \cap W} \Phi_{\tau \cap W} + \sum_{\eta \in \mathcal{B}(\mathfrak{A} \cap W, \mathfrak{c}_{W})} a_{\eta}\Phi_{\eta},
$$
where $a_{\tau \cap W}, a_{\eta} \in \mathbb{R}$. 

If $\beta \in \rho \setminus W$ then $\eta \cup \{\beta\}\in \mathcal{B}(\mathfrak{A})$ and $\Phi_{\rho} = \frac{1}{\beta}\Phi_{\rho \cap W}$, $\Phi_{\eta \cup \{\beta\}} = \frac{1}{\beta} \Phi_{\eta}$, therefore
	\begin{equation}\label{Eq9}
\Phi_{\rho} = a_{\tau \cap W} \Phi_{(\tau \cap W)\cup \{\beta\}} + \sum_{\eta \in \mathcal{B}(\mathfrak{A} \cap W, \mathfrak{c}_{W})} a_{\eta}\Phi_{\eta\cup \{\beta\}}.
	\end{equation}
Since $w \subset \overline{\mathfrak{c}_{W}}\not \subset Cone(\eta)$ we have $w \not \subset Cone(\eta\cup \{\beta\})$, hence $\mathfrak{c} \not \subset  Cone(\eta \cup \{\beta\})$ by Remark \ref{Rk1}. Therefore, $\eta \cup \{\beta\} \in \mathcal{B}(\mathfrak{A},\mathfrak{c})$ for all $\eta \in \mathcal{B}(\mathfrak{A}\cap W,\mathfrak{c}_{W})$.  
Similarly, by Remark \ref{Rk1} we have $\mathfrak{c} \subset  Cone((\tau \cap W) \cup \{\beta\})$, because $w \subset Cone(\tau \cap W)$. 
%{\color{Red}
We apply Lemma \ref{Lem1} for elements of $\{\beta\}\cup \tau$ to get
$$
\Phi_{(\tau\cap W)\cup \{\beta\}} 
= 
b_{\tau}\Phi_{\tau} + \sum_{\sigma \in \mathcal{B}(\mathfrak{A},\mathfrak{c})} b_{\sigma} \Phi_{\sigma}
$$
and in conjuction with (\ref{Eq9}) the lemma follows.
%}
\end{proof}

\begin{proof}[Proof of Proposition \ref{Prop2}]
The proposition is trivial when $\dim V^{*}=1$. Assume that it holds if $\dim V^{*}<r$. We will show that it also holds when $\dim V^{*}=r$. To do so we use wall-crossing: there is a chain of chambers $\mathfrak{c}_{1},\ldots,\mathfrak{c}_{m}=\mathfrak{c}$ in $Cone(\mathfrak{A})$ such that $\mathfrak{c}_{1}$ is as in Lemma \ref{Lem2} and each $\overline{\mathfrak{c}_{i-1}}\cap \overline{\mathfrak{c}_{i}}$ is an $(r-1)$-dimensional polyhedral cone, moreover by Lemma \ref{Lem2} it is enough to show that if the proposition holds for the chamber $\mathfrak{c}_{i-1}$ then it also holds for $\mathfrak{c}_{i}$. 

Remark that we may choose freely the fixed basis $\tau\in \mathcal{B}(\mathfrak{A})$ with $\mathfrak{c} \subset Cone(\tau)$: if $\tau'\in \mathcal{B}(\mathfrak{A})\setminus \mathcal{B}(\mathfrak{A},\mathfrak{c})$ and $\Phi_{\tau'} = a'_{\tau}\Phi_{\tau} + \sum_{\sigma\in \mathcal{B}(\mathfrak{A},\mathfrak{c})} a'_{\sigma}\Phi_{\sigma}$ then $a'_{\tau}\neq 0$ since $\textnormal{JK}_{\varepsilon}(\Phi_{\tau'}) = \textnormal{JK}_{\varepsilon}(a'_{\tau}\Phi_{\tau})\neq0$ for any $\varepsilon\in \mathfrak{c}$.
Let $\tau_{i-1},\tau_{i}\in \mathcal{B}(\mathfrak{A})$ such that $\mathfrak{c}_{i-1}\subset Cone(\tau_{i-1})$ and $\mathfrak{c}_{i}\subset Cone(\tau_{i})$. 

In the case when the inclusion $\mathfrak{c}_{i}\subset Cone(\eta)$ implies $\mathfrak{c}_{i-1} \subset Cone(\eta)$ for all $\eta \in \mathcal{B}(\mathfrak{A})$ we have $\mathcal{B}(\mathfrak{A},\mathfrak{c}_{i}) = \mathcal{B}(\mathfrak{A},\mathfrak{c}_{i-1})$ and we  will choose $\tau_{i}=\tau_{i-1}$. Therefore, the proposition holds for $\mathfrak{c}_{i}$ if it holds for $\mathfrak{c}_{i-1}$. 
We may suppose that there is $\eta \in \mathcal{B}(\mathfrak{A})$ such that $\mathfrak{c}_{i}\subset Cone(\eta)$, but $\mathfrak{c}_{i-1}\not \subset Cone(\eta)$. 
Then $\eta$ has a face which separates chambers $\mathfrak{c}_{i-1}$ and $\mathfrak{c}_{i}$, that is there are $\gamma_{1},\ldots,\gamma_{r-1}\in \eta$ such that the vector space $W$ spanned by them is identical to the supporting hyperplane of the $(r-1)$-dimensional polyhedral cone $w=\overline{\mathfrak{c}_{i-1}}\cap \overline{\mathfrak{c}_{i}}$.
Remark that $w = W\cap \overline{\mathfrak{c}_{i}} \subset Cone(\eta \cap W) = Cone(\gamma_{1},\ldots,\gamma_{r-1})$. 

Since $\mathfrak{c}_{i-1}\subset Cone(\mathfrak{A})$, there is $\gamma_{r+1}\in \mathfrak{A}$ in the same open half-space of $W$ as the chamber $\mathfrak{c}_{i-1}$. By Remark \ref{Rk1} we have $\mathfrak{c}_{i-1}\subset Cone(\gamma_{1},\ldots,\gamma_{r-1},\gamma_{r+1})$, because $w\subset Cone(\gamma_{1},\ldots,\gamma_{r-1})$. If $\gamma_{r}\in \eta \setminus W$ then choose the fixed basis $\tau_{i}$ as $\eta = \{\gamma_{1},\ldots,\gamma_{r} \}$. By Lemma \ref{Lem1} there is a unique $l\in \{1,\ldots,r\}$ such that $\mathfrak{c}_{i}\subset  Cone(\gamma_{1},\ldots,\widehat{\gamma_{l}},\ldots,\gamma_{r},\gamma_{r+1})$. Remark that $l\neq r$ and we may assume that $l=1$. Moreover, $w$ intersects the interior of $Cone(\gamma_{2},\ldots,\gamma_{r+1})$, hence $\mathfrak{c}_{i-1}\subset Cone(\gamma_{2},\ldots,\gamma_{r+1})$, thus we choose $\tau_{i-1}=\{\gamma_{2},\ldots,\gamma_{r+1}\}$. 
By Lemma \ref{Lem1} we have
	\begin{equation}\label{Eq4}
\Phi_{\tau_{i-1}} = b_{\tau_{i}}\Phi_{\tau_{i}} + \sum_{\varrho \in \mathcal{B}(\mathfrak{A},\mathfrak{c}_{i})}b_{\varrho}\Phi_{\varrho}.
	\end{equation}
%
%
%We have a relation between $\Phi_{\tau_{i-1}}$ and $\Phi_{\tau_{i}}$: if $\gamma_{1} = \sum_{i=2}^{r+1}g_{i}\gamma_{i}$ then 
%	\begin{equation}\label{Eq4}
%\Phi_{\tau_{i-1}} = \frac{1}{\gamma_{2}\cdots \gamma_{r+1}} = \frac{g_{2}}{\gamma_{1}\gamma_{3}\cdots \gamma_{r+1}} + \ldots + \frac{g_{r}}{\gamma_{1}\cdots\gamma_{r-1}\gamma_{r+1}} + g_{r+1}\Phi_{\tau_{i}},
%	\end{equation}
%and for all $l=2,\ldots,r$ we have that $\frac{1}{\gamma_{1}\cdots \widehat{\gamma_{l}}\cdots \gamma_{r+1}} \in \mathcal{B}(\mathfrak{A},\mathfrak{c}_{i})$ by Lemma \ref{Lem1}.
Let $\rho\in \mathcal{B}(\mathfrak{A})$ any basis such that $\mathfrak{c}_{i}\subset Cone(\rho)$. If $\mathfrak{c}_{i-1}\subset Cone(\rho)$  then by assumption
	\begin{equation}\label{Eq5}
\Phi_{\rho} = c_{\tau_{i-1}} \Phi_{\tau_{i-1}} + \sum_{\sigma \in \mathcal{B}(\mathfrak{A},\mathfrak{c}_{i-1})} c_{\sigma}\Phi_{\sigma}.
	\end{equation}
Hence by (\ref{Eq4}) and (\ref{Eq5}) we have to only deal with the case when $\mathfrak{c}_{i}\subset Cone(\rho)$, but $\mathfrak{c}_{i-1}\not\subset Cone(\rho)$. As in the above case of $\eta$ we can show that $\rho$ is in the closed half-space $W^{+}$ of $W$ which contains $\mathfrak{c}_{i}$. Then 
the proposition follows from Lemma \ref{Lem2}, because $\tau_{i},\rho\in \mathcal{B}(\mathfrak{A}\cap W^{+})$. 
\end{proof}

\begin{proof}[Proof of Theorem \ref{C}]
For any $\prod_{j\in J}\alpha_{j}$ such that $\varepsilon\notin Cone(\alpha_{i} \,|\, i\in N\setminus J)$ we have 
$$
\textnormal{JK}_{\varepsilon}\left( \frac{ Q\prod_{j\in J}\alpha_{j} }{ \prod_{i\in N} \alpha_{i} }\right) 
= 
\textnormal{JK}_{\varepsilon} \bigg( \frac{Q}{\prod_{i\in N\setminus J}\alpha_{i}}\bigg) 
=
0
$$ 
for any $Q$ by property (P1). Hence, if $P\in \mathcal{I}_{\mathfrak{A},\varepsilon}$ then $\displaystyle\textnormal{JK}_{\varepsilon} \left( \frac{P}{\prod_{i\in N}\alpha_{i}} \right) = 0$.

We prove the other direction by induction on the number of elements $n$ in the list $\mathfrak{A}$. Consider the case $n=r$.
 Since $\varepsilon\in Cone(\mathfrak{A})$ then $\displaystyle\textnormal{JK}_{\varepsilon} \left( \frac{P}{\prod_{i=1}^{r}\alpha_{i}} \right)$ implies that $P=0$, therefore $P\in \mathcal{I}_{\mathfrak{A},\varepsilon}$.
 
 Suppose that the theorem holds when $\mathfrak{A}$ contains less than $n$ vectors. We may also assume that $\varepsilon \in Cone(\alpha_{1},\ldots,\alpha_{r})$, in particular $\{\alpha_{1},\ldots,\alpha_{r}\}$ is a basis of $V^{*}$. Therefore, we can write 
 $$P = \sum_{i=1}^{r} \alpha_{i}P_{i},$$
  where $P_{i}\in \mathbb{R}[V]$ are homogeneous of degree $n-r-1$ for all $i=1,\ldots,r$. We make the assumption that $\varepsilon \in Cone(\alpha_{1},\ldots, \widehat{\alpha_{i}},\ldots,\alpha_{n})$ if $1 \leq i \leq q$ and $\varepsilon \notin Cone(\alpha_{1},\ldots,\widehat{\alpha_{i}},\ldots,\alpha_{n})$ if $q < i \leq r$. Then for each $i\in \{1,\ldots,q\}$ there exists $K_{i}\subset \{1,\ldots,\widehat{i},\ldots,n\}$  such that $\{\alpha_{k} \,|\, k\in K_{i}\}\in \mathcal{B}(\mathfrak{A})$ and $\varepsilon \in Cone(\alpha_{k} \,|\, k\in K_{i})$, that is $\textnormal{JK}_{\varepsilon} \bigg( \dfrac{1}{\prod_{k\in K_{i}}\alpha_{k}} \bigg) \neq 0$. Hence, we can find $c_{i}\in \mathbb{R}$ such that
$$
\textnormal{JK}_{\varepsilon} \left( \frac{\alpha_{i}P_{i}-c_{i}\prod_{l\notin K_{i}}\alpha_{l}}{\prod_{i=1}^{n}\alpha_{i}} \right) 
= 
\textnormal{JK}_{\varepsilon} \left( \frac{\alpha_{i}P_{i}}{\prod_{i=1}^{n}\alpha_{i}} \right) - c_{i} \textnormal{JK}_{\varepsilon} \left( \frac{1}{\prod_{k\in K_{i}}\alpha_{k}} \right) = 0.
$$ 
In particular, $\displaystyle\textnormal{JK}_{\varepsilon} \bigg( \frac{P_{i}-c_{i}\prod_{l\notin K_{i}\cup \{i\}}\alpha_{l}}{\alpha_{1}\cdots \widehat{\alpha_{i}}\cdots \alpha_{n}} \bigg)=0 $, hence if $\mathfrak{A}_{i}$ denotes the list  $\mathfrak{A}$ with $\alpha_{i}$ removed then 
$P_{i} - c_{i}\prod_{l\notin K_{i}\cup \{i\}} \alpha_{l} \in \mathcal{I}_{\mathfrak{A}_{i},\varepsilon}$ by induction hypothesis. 
Moreover, we have $\alpha_{i} \mathcal{I}_{\mathfrak{A}_{i},\varepsilon} \subset \mathcal{I}_{\mathfrak{A},\varepsilon}$, thus
	\begin{equation}\label{Eq6}
\alpha_{i}P_{i} - c_{i} \prod_{l\notin K_{i}} \alpha_{l} \in \mathcal{I}_{\mathfrak{A},\varepsilon}, \qquad\forall 1\leq i \leq q.
	\end{equation}
For $i\in \{q+1,\ldots,r\}$ we have $\alpha_{i}\in \mathcal{I}_{\mathfrak{A},\varepsilon}$, hence 
	\begin{equation}\label{Eq7}
\alpha_{i}P_{i}\in \mathcal{I}_{\mathfrak{A},\varepsilon}, \qquad \forall q<i\leq r.  
	\end{equation}
From (\ref{Eq6}), (\ref{Eq7}) and decomposition 
	\begin{equation}\label{Eq8}
P 
= 
\sum_{i=1}^{q} \Big( \alpha_{i}P_{i} - c_{i}\prod_{l\notin K_{i}}\alpha_{l} \Big) 
+ 
\sum_{i=q+1}^{r} \alpha_{i}P_{i} 
+ 
\sum_{i=1}^{q} c_{i}\prod_{l\notin K_{i}}\alpha_{l}
	\end{equation}
we can see that $P\in \mathcal{I}_{\mathfrak{A},\varepsilon}$ exactly when $\displaystyle Q= \sum_{i=1}^{q}c_{i}\prod_{l\notin K_{i}}\alpha_{l}\in \mathcal{I}_{\mathfrak{A},\varepsilon}$. By Proposition \ref{Prop2}
$$
\frac{Q}{\prod_{i=1}^{n}\alpha_{i}} = \sum_{i=1}^{q} \frac{c_{i}}{\prod_{k\in K_{i}}\alpha_{k}} = \frac{a}{\prod_{i=1}^{r}\alpha_{i}} + \sum_{\sigma \in \mathcal{B}(\mathfrak{A},\mathfrak{c}_{\varepsilon})}\frac{a_{\sigma}}{\prod_{\alpha_{i} \in \sigma}\alpha_{i}},
$$
where $\mathfrak{c}_{\varepsilon}$ is the chamber containing $\varepsilon$. Remark that $\prod\limits_{\alpha_{i} \in \mathfrak{A}\setminus \sigma}\alpha_{i} \in \mathcal{I}_{\mathfrak{A},\varepsilon}$, hence $\displaystyle\textnormal{JK}_{\varepsilon} \left( \frac{1}{\prod_{\alpha_{i}\in \sigma}\alpha_{i}} \right) = 0$ for all $\sigma \in \mathcal{B}(\mathfrak{A},\mathfrak{c}_{\varepsilon})$. Therefore,
\begin{align*}
a \textnormal{JK}_{\varepsilon} \left( \frac{1}{\prod_{i=1}^{r}\alpha_{i}} \right) 
&= 
\textnormal{JK}_{\varepsilon} \left( \frac{a}{\prod_{i=1}^{r}\alpha_{i}} + \sum_{\sigma\in \mathcal{B}(\mathfrak{A},\mathfrak{c}_{\varepsilon})}\frac{a_{\sigma}}{\prod_{\alpha_{i} \in \sigma}\alpha_{i}} \right) 
\\
&= 
\textnormal{JK}_{\varepsilon} \left( \frac{Q}{\prod_{i=1}^{n}\alpha_{i}} \right) = \textnormal{JK}_{\varepsilon} \left( \frac{P}{\prod_{i=1}^{n}\alpha_{i} } \right) 
= 
0,
\end{align*}
by (\ref{Eq6}), (\ref{Eq7}) and (\ref{Eq8}). Moreover, $\displaystyle\textnormal{JK}_{\varepsilon}\left(\frac{1}{\prod_{i=1}^{r}\alpha_{i}}\right) \neq 0$ implies that $a=0$, hence 
$$
Q = \sum_{\sigma \in \mathcal{B}(\mathfrak{A},\mathfrak{c}_{\varepsilon})}a_{\sigma} \prod_{\alpha_{i} \in \mathfrak{A}\setminus\sigma}\alpha_{i} \in \mathcal{I}_{\mathfrak{A},\varepsilon}.
$$

%\newpage

To prove the second part of the theorem, denote $H_{\mathfrak{A},\varepsilon}$ the ideal generated by $\prod_{\alpha_{i}\in H^{+}}\alpha_{i}$ such that $H$ is a hyperplane spanned by a subset of $\mathfrak{A}$ and $H^{+}$ denotes the open half-space containing $\varepsilon$. It is easy to see that $H_{\mathfrak{A},\varepsilon}\subset \mathcal{I}_{\mathfrak{A},\varepsilon}$. To show the reverse inclusion let $\prod_{j\in J}\alpha_{j} \in \mathcal{I}_{\mathfrak{A},\varepsilon}$, i.e. $J\subset N$ is a subset such that $\varepsilon\notin Cone(\alpha_{j}\,|\, j\in N\setminus J)$. We can suppose that $J$ is minimal, hence the latter convex polyhedral cone is at least $(r-1)$-dimensional. Therefore, it has an $(r-1)$-dimensional face such that the underlying hyperplane $H$ separates $\varepsilon$ from this convex polyhedral cone. Moreover, $H$ is spanned by some $\alpha_{i}$'s and $\{\alpha_{j}\,|\, j\in N \setminus J\} \subset \{\alpha_{j}\,|\,\alpha_{j}\notin H^{+}\}$, i.e. $\{\alpha_{j}\,|\, \alpha_{j}\in H^{+} \} \subset \{ \alpha_{j} \,|\, j\in J \}$, hence $\prod_{\alpha_{j}\in H^{+}}\alpha_{j}$ divides $\prod_{j\in J}\alpha_{j}$, therefore $\prod_{j\in J}\alpha_{j}\in H_{\mathfrak{A},\varepsilon}$.
\end{proof}
%%%
\begin{rk*}
	The ideal $\mathcal{I}_{\mathfrak{A},\varepsilon}$ is the image of the Stanley-Reisner ideal $I_{\Sigma}$ of a fan $\Sigma$ under the map $\pi:\mathbb{R}[y_{1},\ldots,y_{n}]\to \mathbb{R}[x_{1},\ldots,x_{r}]$, $y_{i}\mapsto \alpha_{i}$. To construct the fan $\Sigma$ consider the short exact sequence
$$
\xymatrix{0 \ar[r] & \mathbb{R}^{n-r} \ar[r]^{\delta} & \mathbb{R}^{n} \ar[r]^{\gamma} & \mathbb{R}^{r} \ar[r] & 0}
$$
where $\gamma$ sends the standard basis element $e_{i}$ to $\alpha_{i}$. Its Gale dual
$$
\xymatrix{0 \ar[r] & \mathbb{R}^{r} \ar[r]^{\gamma^{*}} & \mathbb{R}^{n} \ar[r]^{\delta^{*}} & \mathbb{R}^{n-r} \ar[r] & 0 }
$$
comes with vectors $\beta_{1},\ldots,\beta_{n}\in \mathbb{R}^{n-r}$ as $\beta_{i}=\delta^{*}(e_{i})$. The fan $\Sigma$ is given as follows: for a subset $J\subset N$ the $Cone(\beta_{j}\,|\, j\in J)\in \Sigma$ if and only if $\varepsilon\in Cone(\alpha_{i}\,|\, i\notin J)$. Since $\varepsilon$ is regular $\Sigma$ is a simplicial fan. To the fan $\Sigma$ we associate the Stanley-Reisner ideal $I_{\Sigma} = \big\langle \prod_{j\in J}y_{j}\ |\ Cone(\beta_{j}\,|\, j\in J)\notin \Sigma \big\rangle \subset \mathbb{R}[y_{1},\ldots,y_{n}]$.
\end{rk*}
\begin{example}
In the $r=2$ case $\mathbb{R}\varepsilon$ is a hyperplane of $V^{*}$ and suppose that vectors $\alpha_{1},\ldots, \alpha_{k}$ and $\alpha_{k+1},\ldots,\alpha_{n}$ lie in different components of $V^{*}\setminus \mathbb{R}\varepsilon$. Then $\mathcal{I}_{\mathfrak{A},\varepsilon}=\left \langle\alpha_{1} \cdots \alpha_{k},\, \alpha_{k+1} \cdots \alpha_{n}\right\rangle$.
\end{example}
Finally, the Main Lemma and Theorem \ref{C} with the remark preceding it give the following result about computation of Jeffrey-Kirwan residues.
	\begin{corollary*}
Let $\mathfrak{A}=[\alpha_{1},\ldots,\alpha_{n}]$ be a collection of non-zero vectors in the $r$-dimensional vector space $V^{*}$. Let $\Lambda$ be a connected component of $V\setminus \bigcup_{i=1}^{n}\alpha_{i}^{\perp}$,  suppose that $\alpha_{1},\ldots,\alpha_{n}\in \Lambda^{\vee}$ and $\varepsilon\in V^{*}$ is regular with respect to $\mathfrak{A}$. Let $J\subset\{1,\ldots,n\}$ be a subset such that $\{ \alpha_{j}\,|\, j\in J\}$ is a basis of $V^{*}$ with $\varepsilon\in Cone(\alpha_{j}\,|\, j\in J)$. 
Then
$$
\textnormal{JKRes}^{\Lambda} \frac{P(x)e^{\varepsilon(x)}}{\prod_{i=1}^{n}\alpha_{i}(x)} \textnormal{d} x  
= 
\textnormal{JK}_{\varepsilon} \left(\frac{P}{\prod_{i=1}^{n}\alpha_{i}} \right) 
= 
\frac{1}{\sqrt{\det[(\alpha_{i},\alpha_{j})]_{i,j\in J}}} \cdot \frac{N_{\mathcal{I}_{\mathfrak{A},\varepsilon}}(P)}{N_{\mathcal{I}_{\mathfrak{A},\varepsilon}}(\prod_{i\notin J}\alpha_{i})} 
$$
for all $P\in \mathbb{R}[V]$ homogeneous of degree $n-r$.
	\end{corollary*}
\noindent
Moreover, $\displaystyle\textnormal{JKRes}^{\Lambda} \frac{P(x) e^{\varepsilon(x)}}{\prod_{i=1}^{n}\alpha_{i}(x)}   \textnormal{d} x$ can be computed with the following algorithm:
{
\renewcommand{\labelenumi}{(\arabic{enumi})}
\begin{enumerate}
\item 
Compute the ideal $\mathcal{I}_{\mathfrak{A},\varepsilon}$ using Theorem \ref{C}.

\item 
Find a $J\subset \{1,\ldots,n\}$ such that $\{\alpha_{j}\, |\, j\in J\}$ is a basis of $V^{*}$ with $\varepsilon\in Cone (\alpha_{j}\, |\, j\in J )$.
\item 
Compute a Gr\"obner basis $G$ of $\mathcal{I}_{\mathfrak{A},\varepsilon}$.

\item 
Compute the remainders (normal forms) $N_{\mathcal{I}_{\mathfrak{A},\varepsilon}}(P)$ and $N_{\mathcal{I}_{\mathfrak{A},\varepsilon}}(\prod_{i\notin J}\alpha_{i})$ of the division of $P$ and $\prod_{i\notin J}\alpha_{i}$ by $G$, respectively.

\item 
Then 
$\displaystyle\textnormal{JKRes}^{\Lambda} \frac{P(x)e^{\varepsilon(x)}}{\prod_{i=1}^{n}\alpha_{i}(x)} \textnormal{d} x  
=  
\frac{1}{\sqrt{\det[(\alpha_{i},\alpha_{j})]_{i,j\in J}}} 
\cdot 
\frac{N_{\mathcal{I}_{\mathfrak{A},\varepsilon}}(P)}{N_{\mathcal{I}_{\mathfrak{A},\varepsilon}}(\prod_{i\notin J}\alpha_{i})}. 
$
\end{enumerate}
}
	\begin{rk*} The remainders can be computed by computer programs like Maple or Macaulay2 as follows.
Let $I$ be an ideal of $\mathbb{R}[x_{1},\ldots,x_{r}]$.
In Maple, first compute the Gr\"obner basis $G$ of $I$:
\begin{verbatim}
with(Groebner): G=Basis(I,tdeg(x_1,...,x_r))
\end{verbatim}
Then compute the remainder of the division of $f$ by $G$:
\begin{verbatim}
NormalForm(f,G,tdeg(x_1,...,x_r))
\end{verbatim}
In Macaulay2 we can compute the remainder directly:
\begin{verbatim}
f % I
\end{verbatim}
	\end{rk*}

%:References

\end{document}